\documentclass{amsart}

\usepackage{amsmath,amssymb,amsthm,amscd}

\usepackage{pinlabel}

\newtheorem{prop}{Proposition}[section]
\newtheorem{thm}[prop]{Theorem}
\newtheorem{lem}[prop]{Lemma}

\theoremstyle{definition}

\newtheorem{ex}[prop]{Example}
\newtheorem{rem}[prop]{Remark}

\newtheorem*{ack}{Acknowledgement}

\def\co{\colon\thinspace}

\newcommand{\ualpha}{\underline{\alpha}}
\newcommand{\talpha}{\tilde{\alpha}}

\newcommand{\ubeta}{\underline{\beta}}

\newcommand{\C}{\mathbb{C}}
\newcommand{\CP}{\mathbb{C}\mathrm{P}}

\newcommand{\rme}{\mathrm{e}}

\newcommand{\Hp}{\mathbb H}
\newcommand{\tildeh}{\tilde{h}}

\newcommand{\rmi}{\mathrm{i}}

\newcommand{\rmj}{\mathrm{j}}

\newcommand{\rmk}{\mathrm{k}}

\newcommand{\tlambda}{\tilde{\lambda}}

\newcommand{\uM}{\underline{M}}
\newcommand{\tmu}{\tilde{\mu}}

\newcommand{\mfp}{\mathfrak{p}}

\newcommand{\R}{\mathbb{R}}
\newcommand{\RP}{\mathbb{R}\mathrm{P}}

\newcommand{\wtT}{\widetilde{T}}

\newcommand{\wtV}{\widetilde{V}}

\newcommand{\oz}{\overline{z}}
\newcommand{\Z}{\mathbb{Z}}

\DeclareMathOperator{\Int}{Int}

\DeclareMathOperator{\lcm}{\mathrm{lcm}}

\DeclareMathOperator{\piorb}{\pi_1^{\mathrm{orb}}}

\begin{document}

\author[H.~Geiges]{Hansj\"org Geiges}

\author[C.~Lange]{Christian Lange}
\address{Mathematisches Institut, Universit\"at zu K\"oln,
Weyertal 86--90, 50931 K\"oln, Germany}
\email{geiges@math.uni-koeln.de, clange@math.uni-koeln.de}

\title{Seifert fibrations of lens spaces}

\date{}

\begin{abstract}
We classify the Seifert fibrations of any given lens space $L(p,q)$.
Starting from any pair of coprime non-zero integers $\alpha_1^0,\alpha_2^0$,
we give an algorithmic construction of a Seifert fibration $L(p,q)\rightarrow
S^2(\alpha|\alpha_1^0|,\alpha|\alpha_2^0|)$, where the natural
number $\alpha$ is determined by
the algorithm. This algorithm produces all possible Seifert fibrations,
and the isomorphisms between the resulting Seifert fibrations are
described completely. Also, we show that all Seifert fibrations
are isomorphic to certain standard models.
\end{abstract}


\thanks{The authors are partially supported by the SFB/TRR 191
`Symplectic Structures in Geometry, Algebra and Dynamics'.}

\maketitle


\section{Introduction}
Seifert fibred $3$-manifolds constitute an important family in the
classification of $3$-manifolds~\cite{scot83}. A Seifert fibration, roughly
speaking, is an $S^1$-fibration with a finite number of multiple fibres;
see Section~\ref{section:seifert} for the precise definition.

It is well known that most $3$-manifolds that admit a Seifert fibration
do so in a unique way, see~\cite[p.~97]{orli72}. The only exceptions
(among closed, orientable $3$-manifolds) are
\begin{enumerate}
\item[(i)] lens spaces (including $S^3$ and $S^2\times S^1$),
\item[(ii)] prism manifolds, and
\item[(iii)] a single euclidean $3$-manifold,
\end{enumerate}
see \cite[Section~5.4 and Chapter~6]{orli72} or
\cite[Theorem~5.1]{jane83}. In (ii) and (iii)
there are two distinct Seifert fibrations per manifold; in (i)
there are infinitely many.

It is not difficult to see that a Seifert fibration of a lens space
can have at most two multiple fibres (Lemma~\ref{lem:two}),
and one can easily compute the diffeomorphism type of the lens
space from the Seifert invariants of such a fibration
(Theorem~\ref{thm:jane4.4}). These results are classical.

However, in the course of our respective works \cite{fls16, lang16} and
\cite{gego15} we noticed a lacuna in the literature on
Seifert fibrations regarding the converse question: given a lens space,
how does one determine all its Seifert fibrations, up to
isomorphism?
Also, how do these fibrations relate under coverings? The answers to
these questions, in certain special cases, were given by
\emph{ad hoc} arguments in the cited papers. This information
was used for geometric applications concerning periodic real Hamiltonian
structures on projective $3$-space, $2$-dimensional
Riemannian orbifolds with all geodesics closed (so-called
Besse orbifolds),
and the moduli theory of contact circles on $3$-manifolds.
In \cite{lang16} a shorter proof for a result of Pries~\cite{prie09}
concerning Besse metrics on the projective plane was given
by appealing to properties of Seifert fibrations;
from a topologist's point of view
this alternative proof is simpler and more natural, but this may be a
matter of taste.

The aim of the present paper is to provide a comprehensive
answer to these questions. In Theorem~\ref{thm:algorithm}
we introduce an algorithm that allows one to produce
a Seifert fibration on a given lens space with arbitrary
prescribed coprime parts of the multiplicities of the two
singular fibres. It is shown that any Seifert fibration arises
in this way, except for two exceptional non-orientable Seifert fibrations
of $L(4,1)$ and $L(4,3)$, for which we exhibit models in
Section~\ref{subsection:non-orientable}.
The isomorphisms between the Seifert fibrations
coming from this algorithm are analysed in Theorem~\ref{thm:equivalences}.

Concerning coverings, in Section~\ref{section:model} we show that every
Seifert fibration of a lens space is isomorphic to a standard model, obtained
by taking a quotient of a standard Seifert fibration of the $3$-sphere.
This is hardly surprising, but along the way we describe a useful
geometric construction for computing the Seifert invariants of these
standard models.

For the background on Seifert manifolds
we only quote results from Seifert's original paper \cite{seif33}
(see the appendix of \cite{seth80} for an English translation)
and from the lecture notes \cite{jane83} by Jankins and Neumann.
With two introductory sections on Seifert manifolds and lens spaces,
respectively, this paper is essentially self-contained.
A further useful reference on Seifert manifolds are the
lecture notes by Brin~\cite{brin07}.

\section{Seifert manifolds}
\label{section:seifert}
In this section we recall the definition of Seifert manifolds
and their classification in terms of Seifert invariants, mostly to set
up our notation.
\subsection{Seifert fibrations}
\label{subsection:seifert}
A \emph{Seifert fibration} of a closed, oriented $3$-manifold $M$
is a smooth map $\pi\co M\rightarrow \Sigma$ onto some
(possibly non-orientable) closed surface $\Sigma$
with the property that any point $x\in\Sigma$
has a neighbourhood $D^2\subset\Sigma$ (with $x=0\in D^2$)
such that $\pi^{-1}(D^2)$ is diffeomorphic to $D^2\times S^1$,
and (for a suitable choice of diffeomorphism)
the map $\pi\co D^2\times S^1\rightarrow D^2$ is given by
\[ \bigl(r\rme^{\rmi\varphi},\rme^{\rmi\theta}\bigr)\longmapsto
r\rme^{\rmi(\alpha\varphi+\alpha'\theta)}\]
for some coprime integers $\alpha,\alpha'$ with $\alpha\neq 0$.
All fibres but the central one $\{0\}\times S^1$ are described
by a pair of equations
\[ r=r_0,\;\;\; \alpha\varphi+\alpha'\theta=\theta_0\]
for some constants $r_0\in (0,1]$ and $\theta_0\in\R$,
where $\theta$ ranges from $0$ to $2\pi\alpha$. The natural
number $|\alpha|$ is called the \emph{multiplicity} of the central fibre;
if $|\alpha|>1$, the central fibre is called \emph{singular}.

In the local model $D^2\times S^1\rightarrow D^2$, all fibres except
perhaps the central one are non-singular. Thus, compactness of $M$
implies that there are only finitely many singular fibres.

Two Seifert fibrations $\pi\co M\rightarrow\Sigma$ and $\pi'\co
M'\rightarrow\Sigma'$ are said to be \emph{isomorphic}
if there is an orientation-preserving diffeomorphism $f\co M\rightarrow M'$
sending fibres to fibres. This induces a diffeomorphism $\overline{f}\co
\Sigma\rightarrow\Sigma'$, giving a commutative diagram
\[ \begin{CD}
M          @>f>>               M'\\
@V{\pi}VV                      @VV{\pi'}V\\
\Sigma     @>{\overline{f}}>>  \Sigma'.
\end{CD} \]
\subsection{Seifert invariants}
For the moment, let us assume that $\Sigma$ is oriented.
Any Seifert fibred $3$-manifold $M\rightarrow\Sigma$
can then be constructed as follows. Consider disjoint model
neighbourhoods around $n\geq 1$ fibres, including
all the singular ones, corresponding to disjoint
discs $D^2_1,\ldots,D^2_n\subset\Sigma$. Set
\[ \Sigma_0=\Sigma\setminus\Int\bigl(D^2_1\sqcup\ldots\sqcup D^2_n\bigr).\]
Over this surface with boundary, the Seifert fibration
restricts to a trivial $S^1$-bundle $M_0=\Sigma_0\times S^1\rightarrow S^1$.

Write the boundary $\partial\Sigma_0$ with the opposite of its natural
orientation as
\[ -\partial\Sigma_0=S^1_1\sqcup\ldots\sqcup S^1_n.\]
In $M_0$ we define the (isotopy classes of) oriented curves
\[ q_i=S^1_i\times\{1\},\; i=1,\ldots,n,\;\;\;\text{and}\;\;\;
h=\{\ast\}\times S^1.\]
Let $V_i=D^2\times S^1$, $i=1,\ldots,n$ be $n$ copies of a solid torus
with respective meridian and longitude
\[ \mu_i=\partial D^2\times\{1\},\;\; \lambda_i=\{1\}\times S^1\subset
\partial V_i.\]

Now, given pairs $(\alpha_i,\beta_i)$, $i=1,\ldots, n$, of coprime
integers with $\alpha_i\neq 0$ one obtains a Seifert fibration
with distinguished fibres of multiplicities $\alpha_1,\ldots,\alpha_n$
(including all the singular fibres)
by gluing the $V_i$ to $M_0$ along the boundary via
the identifications
\begin{equation}
\label{eqn:seifert}
\mu_i=\alpha_i q_i+\beta_i h,\;\;\; \lambda_i=\alpha_i'q_i+\beta_i'h,
\end{equation}
where integers $\alpha_i',\beta_i'$ are chosen such that
\[ \begin{vmatrix}
\alpha_i & \alpha_i'\\
\beta_i  & \beta_i'
\end{vmatrix}=1.\]

Notice that the identifications can equivalently be written as
\[ h=-\alpha_i'\mu_i+\alpha_i\lambda_i,\;\;\;
q_i=\beta_i'\mu_i-\beta_i\lambda_i.\]
Therefore, in the homology of $V_i$ one has $h\sim\alpha_i\lambda_i$
and $q_i\sim -\beta_i\lambda_i$.

With $g$ denoting the genus of~$\Sigma$, the resulting Seifert fibred
manifold $M\rightarrow\Sigma$ is written as
\[ M\bigl( g;(\alpha_1,\beta_1),\ldots,(\alpha_n,\beta_n)\bigr);\]
the integers in this description are called the \emph{Seifert invariants}.

\begin{rem}
The ambiguity in the choice of $\alpha_i',\beta_i'$ is explained as follows.
For $k\in\Z$, let $\tau^k\co V_i\rightarrow V_i$ be the $k$-fold
Dehn twist along the meridional disc $D^{2}\times\{1\}$. This
diffeomorphism of $V_i$ sends $\mu_i$ to itself, and $\lambda_i$
to $\lambda_i+k\mu_i$. Precomposing the gluing map $\partial V_i
\rightarrow\partial M_0$ with $\tau^k|_{\partial V_i}$ amounts
to replacing $\alpha_i',\beta_i'$ by $\alpha_i'+k\alpha_i,\beta_i'+k\beta_i$.

Similarly, the projection map
$\pi\co D^2\times S^1\rightarrow D^2$ in Section~\ref{subsection:seifert}
changes to
\[ \bigl(r\rme^{\rmi\varphi},\rme^{\rmi\theta}\bigr)\longmapsto
r\rme^{\rmi(\alpha\varphi+(\alpha'+k\alpha)\theta)}\]
under precomposition with $\tau^k$.
\end{rem}

The Seifert invariants determine the Seifert fibration up to isomorphism.
Moreover, two sets of Seifert invariants determine isomorphic Seifert
fibrations if and only if one can be changed into the other using the
following operations, see \cite[Theorem~1.5]{jane83}:
\begin{enumerate}
\item[(S0)] Permute the $n$ pairs $(\alpha_i,\beta_i)$.
\item[(S1)] Add or delete any pair $(\alpha,\beta)=(1,0)$.
\item[(S2)] Replace each $(\alpha_i,\beta_i)$ by
$(\alpha_i,\beta_i+k_i\alpha_i)$, where $\sum_{i=1}^n k_i=0$.
\item[(S3)] Replace any $(\alpha_i,\beta_i)$ by $(-\alpha_i,-\beta_i)$.
\end{enumerate}

Reversing the orientation of $M$ amounts to reversing the
orientation of either $h$ or the $q_i$. Thus, replacing each
$(\alpha_i,\beta_i)$ by $(\alpha_i,-\beta_i)$ amounts to
passing to a Seifert fibration of $-M$.

\begin{rem}
\label{rem:sign1}
(1) The operation (S3) corresponds to replacing meridian and longitude
$(\mu_i,\lambda_i)$ of $V_i$
by $(-\mu_i,-\lambda_i)$, which can be effected by a diffeomorphism
of $V_i=D^2\times S^1$ that reverses the orientation of the
$D^2$- and the $S^1$-factor. Usually, it is understood that
the choice is made such that $\alpha_i\geq 1$; for this reason
(S3) does not appear explicitly in \cite[Theorem~1.5]{jane83}.
For us, however, it will be important not to fix the sign of
the~$\alpha_i$, see Remark~\ref{rem:sign2} below.

(2) When there are no singular fibres, one needs $n\geq 1$ to describe
the nontrivial principal $S^1$-bundles over a given surface ~$\Sigma$. For
instance, the Hopf fibration on the $3$-sphere can be described by
the Seifert invariants $\bigl(0;(1,1)\bigr)$, see
Section~\ref{subsection:non-orientable}.

(3) With the help of (S1) to (S4) one can always arrange that
$\alpha_1=1$ and $\alpha_i>1$, $1\leq\beta_i<\alpha_i$ for $i\geq 2$.
We shall not assume, however, that the Seifert invariants have been
normalised in this way. If there is at least one singular fibre,
one can alternatively remove all pairs of the form $(1,\beta)$ by
an application of these equivalences.
\end{rem}

If $\Sigma$ is non-orientable, it can be written as a
connected sum of the real projective plane $\RP^2$ or the
Klein bottle $\RP^2\#\RP^2$ with an orientable surface, and the
singular fibres may be assumed to lie over the orientable part.
The description in terms of Seifert invariants is then as before;
the genus of the base surface is written as a negative number, that is,
$g(\RP^2)=-1$, $g(\RP^2\#\RP^2)=-2$ etc.
\subsection{The fundamental group}
The fundamental group of
\[ M=M\bigl( g;(\alpha_1,\beta_1),\ldots,(\alpha_n,\beta_n)\bigr),\]
as shown in \cite[\S~10]{seif33} or \cite[Section~6]{jane83},
has the presentation
\[ \langle a_1,b_1,\ldots,a_g,b_g,q_1,\dots,q_n,h\,|\,
h\text{ central},\, q_i^{\alpha_i}h^{\beta_i},\,
q_1\cdots q_n[a_1,b_1]\cdots[a_g,b_g]\rangle\]
for $g\geq 0$
and
\[ \langle a_1,\ldots,a_{|g|},q_1,\dots,q_n,h\,|\,
a_j^{-1}ha_j=h^{-1},\, [h,q_i],\,
q_i^{\alpha_i}h^{\beta_i},\,
q_1\cdots q_na_1^2\cdots a_{|g|}^2\rangle\]
for $g< 0$;
here a relation given as a word $w$ is to be read as $w=1$.

Geometrically, one wants to think of the base of a Seifert fibration
$M\rightarrow\Sigma$ as an orbifold with orbifold singularities
of multiplicity $|\alpha_1|,\ldots,|\alpha_n|$. One then writes
$\Sigma(|\alpha_1|,\ldots,|\alpha_n|)$ to indicate the order of the
cone points. The \emph{orbifold fundamental group}
$\piorb(\Sigma)$ is the quotient group of
$\pi_1(M)$ obtained by setting the class $h$ of the
regular fibre equal to~$1$.
\section{Lens spaces}
\label{section:lens}
\subsection{Definition of lens spaces}
For any pair $(p,q)$ of coprime integers with $p>0$,
the \emph{lens space} $L(p,q)$ is the quotient of the $3$-sphere
$S^3\subset\C^2$ under the
free $\Z_p$-action generated by
\begin{equation}
\label{eqn:Z_p}
(z_1,z_2)\longmapsto (\rme^{2\pi\rmi/p}z_1,
\rme^{2\pi\rmi q/p}z_2).
\end{equation}
This lens space inherits a natural orientation from~$S^3$.
Whenever we speak of a diffeomorphism of lens spaces, we mean an
orientation-preserving diffeomorphism of oriented manifolds.

Notice that $L(1,0)=S^3$. The definition of $L(p,q)$ can be
extended to arbitrary coprime integers by setting
$L(0,1):=S^2\times S^1$ and $L(p,q):=L(-p,-q)$ for $p<0$.
This is consistent with the surgery picture that we explain next.
\subsection{Surgery description}
\label{subsection:lens-surgery}
The lens space $L(p,q)$ with its natural orientation can be obtained
from $S^3$ by performing $(-p/q)$-surgery along an unknot,
see~\cite[p.~158]{gost99}. In other words, the lens space $L(p,q)$
is given by gluing two solid tori $V_i=D^2\times S^1$,
$i=1,2$, using the orientation-reversing gluing map
described by
\begin{equation}
\label{eqn:lens-surgery}
\mu_1=-q\mu_2+p\lambda_2,\;\;\; \lambda_1=r\mu_2+s\lambda_2,
\end{equation}
where the integers $r,s$ are chosen such that
\[ \begin{vmatrix}
-q & p\\
r  & s
\end{vmatrix}=-1.\]

Whenever $L(p,q)$ is written as the gluing of two solid tori,
the longitudes $\lambda_1,\lambda_2$ can be chosen such that the gluing map
is as described above. This corresponds with the fact
that $L(p,q)$ depends only on $p$ and the residue class of $q$ modulo~$p$,
and that $(r,s)$ may be changed by multiples of $(-q,p)$.
Furthermore, by exchanging the roles of the two solid tori, one sees that
$L(p,q)$ is diffeomorphic to $L(p,s)$.

As first shown by
Reidemeister~\cite{reid35}, these are the only orientation-preserving
diffeomorphisms between lens spaces, that is, $L(p,q)\cong L(p,q')$
if and only if $q\equiv q'$ or $qq'\equiv 1$ mod~$p$. Similarly, with
$-L(p,q')$ denoting the lens space $L(p,q')$ with the opposite
of its natural orientation, we have $L(p,q)\cong-L(p,q')$ if and only
if $q\equiv -q'$ or $qq'\equiv -1$ mod~$p$.
\section{Seifert fibrations on lens spaces}
\subsection{The model fibrations}
\label{subsection:model}
For any coprime pair $k_1,k_2\in\Z\setminus\{0\}$, the $S^1$-action
\[ \theta(z_1,z_2)=(\rme^{\rmi k_1\theta}z_1,\rme^{\rmi k_2\theta}z_2),\;\;
\theta\in\R/2\pi\Z,\]
on $S^3$ defines a Seifert fibration with singular fibres $S^1\times\{0\}$
and $\{0\}\times S^1$ of multiplicity $k_1$ and $k_2$, respectively.
This $S^1$-action commutes with
the $\Z_p$-action (\ref{eqn:Z_p}) and thus defines a Seifert fibration
on the quotient space $L(p,q)=S^3/\Z_p$.

In Section~\ref{section:model}
we shall prove that --- with the exception of two non-orientable
Seifert fibrations that will be described presently --- any Seifert fibration
of any lens space is isomorphic to one in this standard form. In particular,
we are going to determine the Seifert invariants of these model fibrations.
\subsection{The base of the fibration}
We begin with a simple observation.

\begin{lem}
\label{lem:two}
The base surface of a Seifert fibration of any lens space
is $S^2$ or $\RP^2$. If the base is $S^2$, there are at most
two singular fibres; if the base is $\RP^2$, there is no
singular fibre.
\end{lem}

\begin{proof}
Any loop in the base $\Sigma$ of a Seifert fibration $M\rightarrow\Sigma$
can obviously be lifted to a loop in~$M$; for instance, use an auxiliary
Riemannian metric on $M$ to lift the loop in $\Sigma$ to a path
orthogonal to the fibres in~$M$, then join the endpoints along the
fibre. It follows that the homomorphism $\pi_1(M)\rightarrow
\pi_1(\Sigma)$ is surjective. Thus, if $\pi_1(M)$ is cyclic, then so
is $\pi_1(\Sigma)$. This proves the first statement.

A Seifert fibration
\[ M\bigl(0;(\alpha_1,\beta_1),\ldots,(\alpha_n,\beta_n)\bigr)\]
with base $S^2$ has finite fundamental group only if $n\leq 3$,
see \cite[Satz~9]{seif33}. If $n=3$ (and assuming that
the fundamental group is finite), the quotient group
$\piorb(\Sigma)$ of $\pi_1(M)$ is a platonic group with presentation
\[ \langle q_1,q_2,q_3\;|\;q_1^{\alpha_1},\, q_2^{\alpha_2},\, q_3^{\alpha_3},
\, q_1q_2q_3\rangle,\]
where
\[ (\alpha_1,\alpha_2,\alpha_3)\in\bigl\{(2,2,m), (2,3,3), (2,3,4),
(2,3,5)\bigr\};\]
this case cannot occur if $\pi_1(M)$ is cyclic.

A Seifert fibration
\[ M\bigl(-1;(\alpha_1,\beta_1),\ldots,(\alpha_n,\beta_n)\bigr) \]
with base $\RP^2$ leads to
\[ \piorb(\Sigma)=\langle a,q_1,\ldots,q_n\;|\;q_1^{\alpha_1},\ldots,
q_n^{\alpha_n},\, q_1\cdots q_n a^2\rangle. \]
This group is abelian if and only if $\alpha_i=\pm 1$ for all $i=1,\ldots ,n$.
In that case, using the equivalences (S0)--(S3) for Seifert invariants
described in Section~\ref{section:seifert}, we can pass to an isomorphic
Seifert fibration $M\bigl(-1;(1,b)\bigr)$.
\end{proof}
\subsection{Non-orientable base}
\label{subsection:non-orientable}
Here we deal with the case where the base surface is~$\RP^2$.

\begin{prop}
\label{prop:lens-rp2}
A lens space that fibres over $\RP^2$
is diffeomorphic to $L(4,1)$ or $L(4,3)$.
Each of these lens spaces admits a unique Seifert fibration
with base $\RP^2$.
\end{prop}

\begin{proof}
Write the given fibration as $M=M\bigl(-1;(1,b)\bigr)\bigr)$.
The quotient group
\[ \pi_1\bigl(M\bigl(-1;(1,b)\bigr)\bigr)/\langle q\rangle=
\langle a,h\; |\; a^{-1}ha=h^{-1},\, h^b,\, a^2\rangle \]
is abelian only if $h^2=1$. So we need $b\in\{\pm 1,\pm 2\}$.
For $b=\pm 2$ we have $\pi_1(M)/\langle q\rangle\cong
\Z_2\oplus\Z_2$, which is excluded when $M$ is a lens space.
For $b=\pm 1$ we have
\[ \pi_1(M)\cong\langle a,q,h\; |\; a^{-1}ha=h^{-1},\,
[h,q],\; qh^{\pm 1},\, qa^2\rangle.\]
A straightforward computation reduces this to
\[ \pi_1(M)\cong\langle a\;|\; a^4\rangle\cong\Z_4.\]
Thus, there are at most two potential fibrations of a lens space
over $\RP^2$, and only the lens spaces $L(4,1)$ and $L(4,3)$
might arise in this way.
We now exhibit, on each of these two lens spaces, an $S^1$-fibration 
with base $\RP^2$.

The positive Hopf fibration of the $3$-sphere is the map
\[ \begin{array}{rcl}
\C^2\supset S^3 & \longrightarrow & S^2=\CP^1\\
(z_1,z_2)       & \longmapsto     & [z_1:z_2];
\end{array}\]
this corresponds to the $S^1$-action
$\theta(z_1,z_2)=(\rme^{\rmi\theta}z_1,\rme^{\rmi\theta}z_2)$.
The negative Hopf fibration, corresponding to the $S^1$-action
$\theta(z_1,z_2)=(\rme^{\rmi\theta}z_1,\rme^{-\rmi\theta}z_2)$,
is defined by
\[ (z_1,z_2) \longmapsto [z_1:\oz_2]. \]
Our aim is to lift the antipodal $\Z_2$-action on $S^2$
to a $\Z_4$-action on $S^3$, using either Hopf fibration. Each lift 
induces an $S^1$-bundle $S^3/\Z_4\rightarrow\RP^2$.

\begin{rem}
A consequence of this description is that the resulting
Seifert fibrations of $L(4,1)$ and
$L(4,3)$ are isomorphic to the unit tangent bundle of $\RP^2$
with one or the other orientation.
\end{rem}

First we need to describe the antipodal action in terms of homogeneous
coordinates on $S^2=\CP^1$.
The stereographic projection $\R^3\supset S^2\rightarrow\R^2\equiv\C$
from the north pole $(0,0,1)$ is given by
\[ (x,y,t)\longmapsto \frac{x+\rmi y}{1-t}=:z;\]
the antipodal point is mapped as
\[ (-x,-y,-t)\longmapsto \frac{-(x+\rmi y)}{1+t}=:\psi (z).\]
It follows that $\oz\cdot\psi(z)=-1$. Thus, in homogeneous
coordinates the antipodal map is described by
\[ [z_1:z_2]\longmapsto [\oz_2:-\oz_1].\]

For the positive Hopf fibration, this $\Z_2$-action is covered
by the $\Z_4$-action on $S^3$ generated by
\[ A_+\co (z_1,z_2)\longmapsto (\oz_2,-\oz_1);\]
for the negative Hopf fibration, the lifted action is generated by
\[ A_-\co (z_1,z_2)\longmapsto (z_2,-z_1).\]

On the other hand, the $\Z_4$-actions on $S^3$ producing the quotients
$L(4,1)$ and $L(4,3)$ are generated by
\[ A_1\co (z_1,z_2)\longmapsto (\rmi z_1,\rmi z_2)\]
and
\[ A_3\co (z_1,z_2)\longmapsto (\rmi z_1,-\rmi z_2),\]
respectively.

In quaternionic notation $z_1+z_2\rmj=:a_0+a_1\rmi+a_2\rmj+a_3\rmk=:a\in
S^3\subset\Hp$,
these maps take the simple form
\[ A_+(a)=-\rmj\cdot a,\;\;\;
A_-(a)=-a\cdot\rmj,\;\;\;
A_1(a)=\rmi\cdot a,\;\;\;
A_3(a)=a\cdot\rmi.\]
A straightforward calculation then shows that the map
$\phi\co S^3\rightarrow S^3$ defined by
\[ \phi(a_0+a_1\rmi+a_2\rmj+a_3\rmk):= a_2+a_0\rmi+a_1\rmj+a_3\rmk\]
conjugates these actions as follows:
\[ A_1\circ\phi=\phi\circ A_+,\;\;\;
A_3\circ\phi=\phi\circ A_-.\]
Such conjugating maps can be found with an ansatz $\phi(a)=bac$,
where $b$ and $c$ are unit quaternions. Our choice corresponds to
$b=(1+\rmi-\rmj+\rmk)/2$ and $c=(1+\rmi-\rmj-\rmk)/2$.
\end{proof}
\subsection{Orientable base}
Any Seifert fibration over $S^2$ with at most two singular
fibres has a total space that is obtained by gluing two
solid tori, i.e.\ a lens space. The following theorem from \cite{jane83}
shows how to determine the type of this lens space from the Seifert
invariants. We include the proof since
the argument will be relevant for answering the converse question:
how to determine the Seifert invariants of all Seifert fibrations of
a given lens space.

\begin{thm}[{\cite[Theorem~4.4]{jane83}}]
\label{thm:jane4.4}
The lens space $L(p,q)$ is diffeomorphic to the Seifert manifold
$M\bigl(0;(\alpha_1,\beta_1),(\alpha_2,\beta_2)\bigr)$, provided that
\[ p = \begin{vmatrix}
        \alpha_1 & \alpha_2\\
        -\beta_1 & \beta_2
       \end{vmatrix}
\;\;\;\text{and}\;\;\;
   q = \begin{vmatrix}
        \alpha_1 & \alpha_2'\\
        -\beta_1 & \beta_2'
       \end{vmatrix},\]
where $(\alpha_2',\beta_2')$ is a solution of
\[ \begin{vmatrix}
    \alpha_2 & \alpha_2'\\
    \beta_2  & \beta_2'
   \end{vmatrix}
 =1. \]
\end{thm}

\begin{proof}
In the notation of Section~\ref{section:seifert}, $\Sigma_0$
is an annulus, which gives us the relation $q_2=-q_1$. With
(\ref{eqn:seifert}) we find
\begin{eqnarray*}
\begin{pmatrix}\mu_1\\ \lambda_1\end{pmatrix} & = &
      \begin{pmatrix}
      \alpha_1  & \beta_1\\
      \alpha_1' & \beta_1'
      \end{pmatrix}
      \begin{pmatrix}q_1\\ h\end{pmatrix}\\
& = & \begin{pmatrix}
      -\alpha_1  & \beta_1\\
      -\alpha_1' & \beta_1'
      \end{pmatrix}
      \begin{pmatrix}q_2\\ h\end{pmatrix}\\
& = & \begin{pmatrix}
      -\alpha_1  & \beta_1\\
      -\alpha_1' & \beta_1'
      \end{pmatrix}
      \begin{pmatrix}
      \beta_2'   & -\beta_2\\
      -\alpha_2' & \alpha_2
      \end{pmatrix}
      \begin{pmatrix}\mu_2\\ \lambda_2\end{pmatrix}\\
& = & \begin{pmatrix}
      -(\alpha_1\beta_2'+\beta_1\alpha_2')   &
              \alpha_1\beta_2+\beta_1\alpha_2\\
      -(\alpha_1'\beta_2'+\beta_1'\alpha_2') &
              \alpha_1'\beta_2+\beta_1'\alpha_2
      \end{pmatrix}
      \begin{pmatrix}\mu_2\\ \lambda_2\end{pmatrix}.
\end{eqnarray*}
The theorem follows by comparing this with~(\ref{eqn:lens-surgery}).
\end{proof}

\begin{rem}
Alternatively, but less explicitly than in the proof of
Proposition~\ref{prop:lens-rp2},
one can determine the lens spaces $M\bigl(-1;(1,\pm 1)\bigr)$
by appealing to \cite[Theorem~5.1]{jane83} (whose proof is left
as an exercise with hints). According to that theorem, there is
an orientation-preserving diffeomorphism (\emph{not} an
isomorphism of Seifert fibrations)
\[ M\bigl(-1;(1,\pm 1)\bigr) \cong
M\bigl(0;(2,1),(2,-1),(\mp1,1)\bigr).\]
The latter Seifert fibration is isomorphic to
\[ M\bigl(0;(2,1),(2,-1),(1,\mp 1)\bigr)=
M\bigl(0;(2,\mp 1),(2,\mp 1)\bigr).\]
With Theorem~\ref{thm:jane4.4} one finds that $M$ is diffeomorphic
(as an oriented manifold) to
$L(-4,3)=L(4,-3)=L(4,1)$ for $b=+1$,
and to $L(4,3)$ for $b=-1$.
\end{rem}

Using Theorem~\ref{thm:jane4.4}, we can deal directly with
the lens space $L(0,1)=S^2\times S^1$.

\begin{prop}
A complete list of the Seifert fibrations of $S^2\times S^1$
is provided by
\[ M\bigl(0;(\alpha,\beta),(\alpha,-\beta)\bigr),\]
where $(\alpha,\beta)$ is any pair of coprime integers with
$\alpha>0$ and $\beta\geq 0$.
\end{prop}

\begin{proof}
In a description $L(0,1)=M\bigl(0;(\alpha_1,\beta_1),
(\alpha_2,\beta_2)\bigr)$ we may assume by (S3) that $\alpha_1,\alpha_2>0$.
The equation $0=p=\alpha_1\beta_2+\beta_1\alpha_2$ with
$\gcd(\alpha_i,\beta_i)=1$ is then equivalent to $\alpha_1=\alpha_2$
and $\beta_1=-\beta_2$.
\end{proof}

Of course, for $\beta=0$ we have $\alpha=1$; this corresponds to the
obvious $S^1$-fibration of $S^2\times S^1$.

We now turn the proof of Theorem~\ref{thm:jane4.4} on its head, as it were,
with the aim of determining all Seifert fibrations of
a fixed lens space $L(p,q)$, up to isomorphism,
where from now on $p>0$ is understood.

Given a Seifert fibration
\[ \pi\co L(p,q)=M\bigl(0;(\alpha_1,\beta_1),(\alpha_2,\beta_2)\bigr)
\longrightarrow S^2,\]
choose a circle $C\subset S^2$ separating the (at most) two orbifold points,
and decompose $L(p,q)$ into two solid tori $V_1,V_2$ along $\pi^{-1}(C)$.
As explained in Section~\ref{subsection:lens-surgery}, we can choose
longitudes on $\partial V_1$ and $\partial V_2$ such that the
gluing map of the two solid tori is given by~(\ref{eqn:lens-surgery}).
On the other hand, $L(p,q)$ may also be thought of as being
obtained by gluing $V_1,V_2$ to a thickening of the torus $\pi^{-1}(C)$,
using the identifications~(\ref{eqn:seifert}).
We then have
\[ \begin{pmatrix}-q_2\\ h\end{pmatrix}=
\begin{pmatrix}
-\beta_2'  & \beta_2\\
-\alpha_2' & \alpha_2
\end{pmatrix}
\begin{pmatrix}\mu_2\\ \lambda_2\end{pmatrix}\]
and
\[ \begin{pmatrix}q_1\\ h\end{pmatrix}=
\begin{pmatrix}
\beta_1'   & -\beta_1\\
-\alpha_1' & \alpha_1
\end{pmatrix}
\begin{pmatrix}\mu_1\\ \lambda_1\end{pmatrix}=
\begin{pmatrix}
\beta_1'   & -\beta_1\\
-\alpha_1' & \alpha_1
\end{pmatrix}
\begin{pmatrix}
-q & p\\
r  & s
\end{pmatrix}
\begin{pmatrix}\mu_2\\ \lambda_2\end{pmatrix}.\]

\begin{rem}
\label{rem:sign2}
Once the gluing map (\ref{eqn:lens-surgery}) is given, we
are no longer free to replace only one of $(\mu_i,\lambda_i)$ by
$(-\mu_i,-\lambda_i)$. For this reason, we may not fix the
signs of $\alpha_1$ and $\alpha_2$ simultaneously, cf.\
Remark~\ref{rem:sign1}.
\end{rem}

By expressing the relation $(q_1,h)=(-q_2,h)$ in terms of
$(\mu_2,\lambda_2)$, we arrive at the identities
\begin{equation}
\label{eqn:relations}
\begin{array}{rcl}
\alpha_2  & = & s\alpha_1-p\alpha_1',\\[1mm]
\alpha_2' & = & -r\alpha_1-q\alpha_1',\\[1mm]
\beta_2   & = & -s\beta_1+p\beta_1',\\[1mm]
\beta_2'  & = & r\beta_1+q\beta_1'.
\end{array}
\end{equation}

Set $\alpha=\gcd(\alpha_1,\alpha_2)$ and $\alpha_i^0=\alpha_i/\alpha$,
so that $\gcd(\alpha_1^0,\alpha_2^0)=1$.
The first equation above can then be written as
\begin{equation}
\label{eqn:alpha}
p\alpha_1'=\alpha (s\alpha_1^0-\alpha_2^0).
\end{equation}

We first deal with the special case $s\alpha_1^0=\alpha_2^0$.
Then, since $p>0$, we have $\alpha_1'=0$, and hence $\alpha_1\beta_1'=1$.
Without loss of generality we may assume $\alpha_1=1$. This gives
$\beta_1'=1$ and, from the first equation of~(\ref{eqn:relations}),
$\alpha_2=s$. By applying (S2) we may assume $\beta_1=0$.
The third equation of (\ref{eqn:relations}) then gives $\beta_2=p$.
With (S1) we may remove the pair $(\alpha_1,\beta_1)=(1,0)$,
leaving us with a Seifert fibration $M\bigl(0;(s,p)\bigr)$.
Notice that  $qs\equiv 1$ mod $p$. Conversely, one checks
easily with Theorem~\ref{thm:jane4.4} that this condition on $s$
guarantees that the resulting lens space is $L(p,q)$.
Since we may reverse the roles of $q$ and $s$ as described in
Section~\ref{subsection:lens-surgery}, this proves the following.

\begin{prop}
\label{prop:one}
A complete list of the Seifert fibrations of $L(p,q)$, $p>0$,
over $S^2$ with at most one singular fibre is given by
\[ M\bigl(0;(\alpha_2,p)\bigr),\]
where $\alpha_2$ is any non-zero integer with $\alpha_2\equiv q$
or $\alpha_2q\equiv 1$ mod $p$.
\qed
\end{prop}

\begin{ex}
The Seifert fibrations of $S^3=L(1,0)$ with at most one
singular fibre are given by the $S^1$-actions
\[ \theta(z_1,z_2)=(\rme^{\rmi\theta}z_1,\rme^{\rmi s\theta}z_2),\]
where $s$ can be any non-zero integer.
\end{ex}

Since $\alpha_1$ and $\alpha_1'$ are coprime, we also have
$\gcd(\alpha,\alpha_1')=1$. If $s\alpha_1^0\neq\alpha_2^0$,
then from (\ref{eqn:alpha}) we have
\begin{equation}
\label{eqn:defn1}
\alpha= \frac{p}{\gcd(p,s\alpha_1^0-\alpha_2^0)}
\end{equation}
and, likewise,
\begin{equation}
\label{eqn:defn2}
\alpha_1'=\frac{s\alpha_1^0-\alpha_2^0}{\gcd(p,s\alpha_1^0-\alpha_2^0)}.
\end{equation}
Conversely, if one defines $\alpha$ and $\alpha_1'$ by these equations,
then (\ref{eqn:alpha}) holds.
\subsection{An algorithm for finding Seifert fibrations}
We now use this to determine all Seifert fibrations of a given
lens space. Notice that the right-hand sides of equations~(\ref{eqn:defn1})
and~(\ref{eqn:defn2})
make sense also if $s\alpha_1^0=\alpha_2^0$.
The following theorem says, in particular, that for a given lens space
one can always find a Seifert fibration where the coprime parts of the
multiplicities of the singular fibres can be prescribed arbitrarily.
For instance, there is a Seifert fibration
$S^3\rightarrow S^2(k_1,k_2)$ for any pair of
coprime non-zero integers $k_1,k_2$, see Section~\ref{subsection:model}.

\begin{thm}
\label{thm:algorithm}
(i) Let a lens space $L(p,q)$, $p>0$, and a pair of coprime non-zero
integers $\alpha_1^0,\alpha_2^0$
be given. Then there is a Seifert fibration
\[ L(p,q)=M\bigl(0;(\alpha_1,\beta_1),(\alpha_2,\beta_2)\bigr)
\longrightarrow S^2(|\alpha_1|,|\alpha_2|),\]
where the Seifert invariants are defined as follows.
Choose integers $r,s$ such that
\[ \begin{vmatrix}
-q & p\\
r  & s
\end{vmatrix}=-1.\]
Then, with $\alpha$ defined by equation~(\ref{eqn:defn1}), set
\begin{equation}
\label{eqn:defn3}
\alpha_1:=\alpha\alpha_1^0,\;\;\; \alpha_2:=\alpha\alpha_2^0,
\end{equation}
so that $\alpha=\gcd(\alpha_1,\alpha_2)$. Define $\alpha_1'$
by~(\ref{eqn:defn2}). Then there are integers
$\beta_1,\beta_1'$ such that
\begin{equation}
\label{eqn:defn4}
\begin{vmatrix}
\alpha_1 & \alpha_1'\\
\beta_1  & \beta_1'
\end{vmatrix}=1.
\end{equation}
Finally, define $\beta_2$ by the third equation of
(\ref{eqn:relations}), that is,
\begin{equation}
\label{eqn:defn5}
\beta_2:=-s\beta_1+p\beta_1'.
\end{equation}

(ii) The resulting Seifert fibration of $L(p,q)$ in this construction
is independent of the specific choice of $r,s,\beta_1,\beta_1'$. 

(iii) Any Seifert fibration of $L(p,q)$ can be obtained in this way.
\end{thm}

\begin{proof}
(i) We first need to show that $\alpha_1$ and $\alpha_1'$ are coprime.
By the definitions (\ref{eqn:defn1}) and (\ref{eqn:defn2}) of
$\alpha$ and $\alpha_1'$, respectively, we have $\gcd(\alpha,\alpha_1')=1$.
Any divisor of both $\alpha_1$ and $\alpha_1'$ would also,
by~(\ref{eqn:alpha}), divide~$\alpha_2$, and hence $\alpha$.
This validates the definition of $\beta_1$ via (\ref{eqn:defn4}).

Next we want to verify, with the help of
Theorem~\ref{thm:jane4.4}, that $M\bigl(0;(\alpha_1,\beta_1),
(\alpha_2,\beta_2)\bigr)$ is diffeomorphic to $L(p,q)$.
Set $u:=\gcd(p,s\alpha_1^0-\alpha_2^0)$, so that $p=u\alpha$
and $s\alpha_1^0-\alpha_2^0=u\alpha_1'$. Then
\begin{eqnarray*}
\alpha_1\beta_2+\beta_1\alpha_2
  & = & p\alpha_1\beta_1'-s\alpha_1\beta_1+\beta_1\alpha_2\\
  & = & p\alpha_1\beta_1'-\beta_1(s\alpha_1-\alpha_2)\\
  & = & p\alpha_1\beta_1'-\beta_1\alpha(s\alpha_1^0-\alpha_2^0)\\
  & = & p\alpha_1\beta_1'-\beta_1\alpha u\alpha_1'\\
  & = & p(\alpha_1\beta_1'-\beta_1\alpha_1')\\
  & = & p.
\end{eqnarray*}
In particular, this shows that any divisor $d$ of both $\alpha_2$ and
$\beta_2$ also divides~$p$, and hence also $d|s\beta_1$ by~(\ref{eqn:defn5}).
But $\gcd(p,s)=1$, so that $d|\beta_1$. From $p\alpha_1'=s\alpha_1-\alpha_2$
we also have $d|\alpha_1$. Since $\gcd(\alpha_1,\beta_1)=1$, this
forces $d=\pm 1$, that is, $\gcd(\alpha_2,\beta_2)=1$. So $(\alpha_2,\beta_2)$
is indeed an allowable Seifert invariant.

Now choose integers $\alpha_2',\beta_2'$ such that
$\alpha_2\beta_2'-\beta_2\alpha_2'=1$ as in Theorem~\ref{thm:jane4.4}.
We then need to verify that $\alpha_1\beta_2'+\beta_1\alpha_2'\equiv q$
mod~$p$. To this end, it suffices to show that
$s(\alpha_1\beta_2'+\beta_1\alpha_2')\equiv 1$ mod~$p$.

Computing modulo~$p$, from (\ref{eqn:defn5}) we have $\beta_2\equiv
-s\beta_1$. The defining equation for $\alpha_2',\beta_2'$ then becomes
\[ \alpha_2\beta_2'+s\beta_1\alpha_2'\equiv 1.\]
Hence
\begin{eqnarray*}
s(\alpha_1\beta_2'+\beta_1\alpha_2')
  & \equiv & s\alpha_1\beta_2'+1-\alpha_2\beta_2'\\
  & =      & (s\alpha_1-\alpha_2)\beta_2'+1\\
  & =      & p\alpha_1'\beta_2'+1\\
  & \equiv & 1.
\end{eqnarray*}

(ii) The value of $s$ may be changed by adding $kp$, $k\in\Z$. This does
not affect $\gcd(p,s\alpha_1^0-\alpha_2^0)$, hence
$\alpha,\alpha_1,\alpha_2$ remain unchanged. The value of $\alpha_1'$
changes by $k\alpha_1$. With $\beta_1$ unchanged, we need to add
$k\beta_1$ to $\beta_1'$. In total, this leaves $\beta_2$ unchanged.

The pair $(\beta_1,\beta_1')$ may be changed by adding
$k(\alpha_1,\alpha_1')$, $k\in\Z$. This adds $k(-s\alpha_1+p\alpha_1')=
-k\alpha_2$ to $\beta_2$. By rule (S2), this gives an
isomorphic Seifert fibration.

(iii) This statement follows from the decomposition of $L(p,q)$
into two solid tori about the singular fibres, since all the
defining identities were derived from such a composition.
Exchanging the roles of the two solid tori, which amounts to
passing from $L(p,q)$ to the diffeomorphic $L(p,s)$, is the same
as exchanging the role of $\alpha_1^0$ and~$\alpha_2^0$.
\end{proof}

\begin{ex}
With $\alpha_1^0=1$ and $q\alpha_2^0\equiv 1$ mod~$p$, the
algorithm in the theorem yields the Seifert fibration
$M\bigl(0;(1,0),(\alpha_2^0,p)\bigr)$ of $L(p,q)$. With $\alpha_1^0\equiv q$
mod~$p$ and $\alpha_2^0=1$ one obtains $M\bigl(0;(\alpha_1^0,p),(1,0)\bigr)$.
These are the Seifert fibrations described in
Proposition~\ref{prop:one}.
\end{ex}

\subsection{Isomorphisms between Seifert fibrations}
We now analyse when the algorithm in Theorem~\ref{thm:algorithm}
leads to isomorphic Seifert fibrations of the lens space
$L(p,q)$. A necessary condition
is of course that the multiplicities of the singular fibres must
coincide, so we fix an unordered pair $\{|\alpha_1^0|,
|\alpha_2^0|\}$.

The main result of this section is the following.

\begin{thm}
\label{thm:equivalences}
Fix a lens space $L(p,q)$, $p>0$, and a pair of coprime natural numbers
$\{|\alpha_1^0|,|\alpha_2^0|\}$. We consider Seifert fibrations
of the form
\[ L(p,q)\longrightarrow S^2(m_1,m_2),\]
where the coprime parts
\[ \Bigl\{\frac{m_1}{\gcd(m_1,m_2)},\frac{m_2}{\gcd(m_1,m_2)}\Bigr\} \]
equal $\{|\alpha_1^0|,|\alpha_2^0|\}$.

(i) If $|\alpha_1^0|=|\alpha_2^0|=1$, then there are precisely two
distinct such Seifert fibrations of $L(p,q)$.
There is an orientation-reversing
isomorphism between these two fibrations if and only if $q^2\equiv -1$
mod~$p$.

(ii) If $|\alpha_1^0|\neq|\alpha_2^0|$, then the following holds:
\begin{enumerate}
\item[(1)] If $q^2\not\equiv\pm 1$ mod~$p$, there are exactly
four distinct such Seifert fibrations.
\item[(2)] If $q^2\equiv 1$ but $q^2\not\equiv -1$, there are exactly
two distinct such Seifert fibrations.
\item[(3)] If $q^2\equiv -1$ but $q^2\not\equiv 1$, there are exactly
four such Seifert fibrations, which are orientation-reversingly
isomorphic in two pairs.
\item[(4)] If $q^2\equiv\pm 1$ (and hence $p\in\{1,2\}$) there
are exactly two such Seifert fibrations, which are
orientation-reversingly isomorphic.
\end{enumerate}
\end{thm}

The proof of this theorem will take up the remainder of this section.
We need to study the Seifert fibrations constructed with the
algorithm in Theorem~\ref{thm:algorithm}, starting from the
ordered pairs $(\alpha_1^0,\alpha_2^0)$, $(\alpha_1^0,-\alpha_2^0)$,
$(\alpha_2^0,\alpha_1^0)$ and $(\alpha_2^0,-\alpha_1^0)$.
Changing the sign of both $\alpha_1^0$ and $\alpha_2^0$ has no effect:
in the algorithm both the signs of the $\alpha_i$ and the $\beta_i$ become
reversed.

\vspace{2mm}

\textbf{(A)} As a first case, we consider the Seifert fibration
\[ M:=M\bigl(0;(\alpha_1,\beta_1),(\alpha_2,\beta_2)\bigr)\]
of $L(p,q)$ obtained from the pair $(\alpha_1^0,\alpha_2^0)$, and
\[ \uM:=M\bigl(0;(\ualpha_1,\ubeta_1),(\ualpha_2,\ubeta_2)\bigr),\]
corresponding to $(\alpha_1^0,-\alpha_2^0)$. Our notational
convention is that all quantities corresponding to $\uM$
will be underlined. We want to decide when $M=\uM$, by which we mean
isomorphism of Seifert fibrations, or $M=-\uM$.

The latter means that there exists an orientation-\emph{reversing}
diffeomorphism between $M$ and $\uM$ sending fibres to fibres.
Of course, this can only happen if $L(p,q)$ admits an
orientation-reversing diffeomorphism, which by
Section~\ref{subsection:lens-surgery} is equivalent to $2q\equiv 0$
(i.e.\ $p\in\{1,2\}$) or $q^2\equiv -1$ mod~$p$; the first case
is subsumed by the second one. Recall from
Section~\ref{subsection:seifert} that 
\[ -\uM=M\bigl(0;(\ualpha_1,-\ubeta_1),(\ualpha_2,-\ubeta_2)\bigr).\]

The condition $|\alpha_1|=|\ualpha_1|$, $|\alpha_2|=|\ualpha_2|$ on the
multiplicities translates into $\alpha=\ualpha$, or
\begin{equation*}
\label{eqn:u-A}
\tag{u-A}
u:=\gcd(p,s\alpha_1^0-\alpha_2^0)=\gcd(p,s\alpha_1^0+\alpha_2^0).
\end{equation*}
Thus, $\ualpha_1=\alpha_1$, $\ualpha_2=-\alpha_2$, i.e.
\[ \uM=M\bigl(0;(\alpha_1,\ubeta_1),(-\alpha_2,\ubeta_2)\bigr)
=M\bigl(0;(\alpha_1,\ubeta_1),(\alpha_2,-\ubeta_2)\bigr).\]

\vspace{1mm}

\textbf{(A.1)}
Assume that $|\alpha_1^0|\neq |\alpha_2^0|$, so that the two singular
fibres can be distinguished by their multiplicity. Then $M=\pm\uM$ is
equivalent to the existence of an $\ell\in\Z$ such that
\[ \pm\ubeta_1=\beta_1+\ell\alpha_1,\;\;\; \mp\ubeta_2=
\beta_2-\ell\alpha_2.\]
For $\uM$ to come from our algorithm, the defining equations
in Theorem~\ref{thm:algorithm} must be satisfied. Using
(\ref{eqn:defn2}) and the third
equation of (\ref{eqn:relations}), we can express
condition (\ref{eqn:defn4}) in terms of unprimed quantities as follows:
\[ p=p\begin{vmatrix}\alpha_1&\alpha_1'\\ \beta_1&\beta_1'\end{vmatrix}=
\alpha_1(\beta_2+s\beta_1)-\beta_1(s\alpha_1-\alpha_2)
=\alpha_1\beta_2+\beta_1\alpha_2.\]
For the underlined quantities we get
\begin{eqnarray*}
p\begin{vmatrix}\ualpha_1&\ualpha_1'\\ \ubeta_1&\ubeta_1'\end{vmatrix}
  & = & \alpha_1(\mp\beta_2\pm\ell\alpha_2\pm s\beta_1\pm\ell s\alpha_1)
        \mp(\beta_1+\ell\alpha_1)(s\alpha_1+\alpha_2)\\
  & = & \mp(\alpha_1\beta_2+\beta_1\alpha_2)\\
  & = & \mp p.
\end{eqnarray*}
This means that only the lower choice of sign is
possible, so at best we might have $M=-\uM$.

For this to be the case, we need only ensure
that $\ubeta_1'=(\ubeta_2+s\ubeta_1)/p$ is actually integral, that is,
we have the divisibility condition
\[p\,|\,(\ubeta_2 + s \ubeta_1)=
\bigl(\beta_2-s\beta_1-\ell(s\alpha_1+\alpha_2)\bigr). \]
Since $u$ divides $s\alpha_1^0+\alpha_2^0$, and $p=u\alpha$, this
condition reduces to
\[ p\,|\,(\beta_2-s\beta_1).\]
But of course we also have $p|(\beta_2+s\beta_1)$ from the
fibration~$M$, so $p=u\alpha$ must divide~$2\beta_1$.

With $\gcd(\alpha,\beta_1)=1$ this gives $\alpha\in\{1,2\}$.
Moreover, from the defining equation (\ref{eqn:u-A})
we have, with $\gcd(s,p)=1$, that $u$ divides both $2\alpha_1^0$
and $2\alpha_2^0$. With $\gcd(\alpha_1^0,\alpha_2^0)=1$ this
yields $u\in\{1,2\}$. Hence $p\in\{1,2,4\}$.

For $p$ equal to $1$ or~$2$, the divisibility condition is
obviously satisfied. For $p=4$, i.e.\ $u=\alpha=2$,
the integer $\beta_1$ would have to be even, and hence $\alpha_1$ odd,
contradicting $\alpha|\alpha_1$; so this is excluded.

Observe that for $p\in\{1,2\}$, condition (\ref{eqn:u-A}) is
automatically satisfied.

We now want to argue that, conversely, the condition $p\in\{1,2\}$,
which guarantees the divisibility condition, is also sufficient
for $M=-\uM$ to hold. (This part of the argument is analogous
in the cases below, and will not be repeated there.) Indeed,
if we \emph{define} $\ubeta_1=-\beta_1-\ell\alpha_1$ and
$\ubeta_2=\beta_2-\ell\alpha_2$ for some $\ell\in\Z$, the divisibility
condition allows us to define $\ubeta_1'$ by
$\ubeta_1'=(\ubeta_2+s\ubeta_1)/p$. Also, we define
$\ualpha_1'$ by the underlined version of (\ref{eqn:defn2}),
that is, $\ualpha_1'=(s\alpha_1^0+\alpha_2^0)/u$. The computations above
then show that all the defining equations in the algorithm
of Theorem~\ref{thm:algorithm} are satisfied. Since $\ubeta_1$
is determined in this algorithm up to adding integer multiples of
$\ualpha_1$, this shows that any possible choice in the
construction of $\uM$ arises in the way just described.

\vspace{1mm}

\textbf{(A.2)}
If $|\alpha_1^0|=|\alpha_2^0|$ (and hence equal to~$1$), write
$\alpha_2=\varepsilon\alpha_1$ with $\varepsilon\in\{\pm 1\}$.
In addition to the options
for $M=\pm\uM$ discussed under \textbf{(A.1)}, we have the
freedom to exchange the roles of the two singular fibres.
This translates into
\[ \pm\varepsilon\ubeta_1=\beta_2+\ell\alpha_2,\;\;\;
\mp\varepsilon\ubeta_2=\beta_1-\ell\alpha_1.\]
We compute
\begin{eqnarray*}
p\begin{vmatrix}
\ualpha_1&\ualpha_1'\\
\ubeta_1 &\ubeta_1'
\end{vmatrix}
  & = & \alpha_1(\ubeta_2+s\ubeta_1)-\ubeta_1(s\alpha_1+\alpha_2)\\
  & = & \varepsilon\alpha_1(\mp\beta_1\pm\ell\alpha_1\pm s\beta_2
        \pm\ell s\alpha_2)
        \mp\varepsilon(\beta_2+\ell\alpha_2)(s\alpha_1+\alpha_2)\\
  & = & \mp(\alpha_2\beta_1+\beta_2\alpha_1)\\
  & = & \mp p.
\end{eqnarray*}
Once again, by~(\ref{eqn:defn4}), only $M=-\uM$ is possible.

The divisibility condition now becomes
\[ p\,|\,(\ubeta_2+s\ubeta_1)=\bigl(\varepsilon(\beta_1-s\beta_2)-
\ell(s\alpha_1+\alpha_2)\bigr),\]
which reduces to
$p|(\beta_1-s\beta_2)$.
With the third equation from (\ref{eqn:relations}) this is equivalent to
\[ p\,|\,\beta_1(1+s^2).\]
With $p=u\alpha$ and $\gcd(\alpha,\beta_1)=1$ this implies $\alpha|(1+s^2)$.

Now $u\in\{1,2\}$ as in \textbf{(A.1)}. For $u=1$
the divisibility condition is equivalent
to $p|(1+s^2)$. Having $u=2$ means
\[ 2=\gcd(p,s-1)=\gcd(p,s+1)\]
by (\ref{eqn:u-A}). But then one of $s-1$ or $s+1$ is divisible
by~$4$, which means that $p$ is not. So $\alpha$ must be odd.
Then the divisibility condition is again equivalent to
$p=2\alpha$ being a divisor of $1+s^2$.

Notice that the condition $s^2\equiv-1$ mod~$p$ is equivalent to
$q^2\equiv-1$. Also, this condition is automatically satisfied for
$p\in\{1,2\}$, so we need not list these as separate options coming
from~\textbf{(A.1)}.

Observe that (\ref{eqn:u-A}) is again
a consequence of $s^2\equiv-1$, since
\begin{eqnarray*}
\gcd(p,(s-\varepsilon)\alpha_1^0)
  & = & \gcd(p,s(s-\varepsilon)\alpha_1^0)\\
  & = & \gcd(p,-(1+s\varepsilon)\alpha_1^0)\\
  & = & \gcd(p,(s+\varepsilon)\alpha_1^0).
\end{eqnarray*}

\vspace{1mm}

We summarise case \textbf{(A)} in the following proposition.

\begin{prop}
\label{prop:A}
When the pair $(\alpha_1^0,\alpha_2^0)$ is replaced by
$(\alpha_1^0,-\alpha_2^0)$ in the algorithm of Theorem~\ref{thm:algorithm},
the two resulting Seifert fibrations on $L(p,q)$ are never isomorphic as
oriented fibrations.
An orientation-reversing isomorphism
exists, for $|\alpha_1^0|\neq|\alpha_2^0|$, precisely when
$p\in\{1,2\}$; for $|\alpha_1^0|=|\alpha_2^0|=1$, if and only if
$q^2\equiv-1$ mod~$p$.
\qed
\end{prop}

\vspace{2mm}

\textbf{(B)} Again we write $M$ for the Seifert fibration
corresponding to the pair $(\alpha_1^0,\alpha_2^0)$;
the Seifert fibration $\uM$
is now taken to be the one
coming from $(\alpha_2^0,\alpha_1^0)$. As before, for $M$ and $\uM$
(or $-\uM$) to be isomorphic we need $\alpha=\ualpha$, which translates into
\begin{equation*}
\label{eqn:u-B}
\tag{u-B}
u:=\gcd(p,s\alpha_1^0-\alpha_2^0)=\gcd(p,s\alpha_2^0-\alpha_1^0).
\end{equation*}
We can then write
\[ \uM=M\bigl(0;(\alpha_2,\ubeta_1),(\alpha_1,\ubeta_2)\bigr).\]

\vspace{1mm}

\textbf{(B.1)} For $|\alpha_1^0|\neq|\alpha_2^0|$, the statement
$M=\pm\uM$ is equivalent to the existence of an $\ell\in\Z$ such that
\[ \pm\ubeta_1=\beta_2-\ell\alpha_2,\;\;\;
\pm\ubeta_2=\beta_1+\ell\alpha_1.\]
From
\begin{eqnarray*}
p\begin{vmatrix}\ualpha_1&\ualpha_1'\\ \ubeta_1&\ubeta_1'\end{vmatrix}
  & = & \alpha_2(\pm\beta_1\pm\ell\alpha_1\pm s\beta_2\mp\ell s\alpha_2)
        \mp(\beta_2-\ell\alpha_2)(s\alpha_2-\alpha_1)\\
  & = & \pm (\alpha_1\beta_2+\beta_1\alpha_2)\\
  & = & \pm p
\end{eqnarray*}
we see that only $M=\uM$ is an option.

The divisibility condition
\[ p\,|\, (\ubeta_2+s\ubeta_1)=\bigl((\beta_1+s\beta_2)-\ell
(s\alpha_2-\alpha_1)\bigr)\]
is equivalent to $p|(\beta_1+s\beta_2)$; with the third equation
of (\ref{eqn:relations}) this becomes
\[ p\,|\,\beta_1(1-s^2).\]

From (\ref{eqn:u-B}) we see that $u$ divides both
\[ s\alpha_2^0-\alpha_1^0+s(s\alpha_1^0-\alpha_2^0)=
-(1-s^2)\alpha_1^0\]
and
\[ s\alpha_1^0-\alpha_2^0+s(s\alpha_2^0-\alpha_1^0)=
-(1-s^2)\alpha_2^0.\]
Since $\alpha_1^0$ and $\alpha_2^0$ are coprime, this means $u|(1-s^2)$.
From
\[ p=u\alpha\,|\,\beta_1(1-s^2),\;\;\; u\,|\,(1-s^2),\;\;\;\text{and}\;\;\;
\gcd(\alpha,\beta_1)=1\]
we find $p|(1-s^2)$, so the divisibility condition is $q^2\equiv 1$
mod~$p$.

The requirement (\ref{eqn:u-B}) is implied by this condition,
since
\[ s(s\alpha_1^0-\alpha_2^0)\equiv-(s\alpha_2^0-\alpha_1^0)\;\text{mod}\;p.\]

\vspace{1mm}

\textbf{(B.2)} If $|\alpha_1^0|=|\alpha_2^0|=1$, and with
$\alpha_2=\varepsilon\alpha_1$, $\varepsilon\in\{\pm 1\}$, we have
the additional possibility that
\[ \pm\varepsilon\ubeta_1=\beta_1+\ell\alpha_1,\;\;\;
\pm\varepsilon\ubeta_2=\beta_2-\ell\alpha_2.\]
Again one checks that only $M=\uM$ might happen. The divisibility
condition now becomes $p|(\beta_2+s\beta_1)$, which is always satisfied
by~(\ref{eqn:relations}).

This should not come as a surprise: reversing the roles of the two
singular fibres simply cancels the effect of exchanging the
two multiplicities (of equal absolute value).

Also, condition (\ref{eqn:u-B}) is again empty, since
\[ s(s-\varepsilon)\equiv -(s\varepsilon-1)\;\text{mod}\; p.\]

\vspace{1mm}

Summarising, we have the following statement.

\begin{prop}
\label{prop:B}
The two Seifert fibrations coming from $(\alpha_1^0,\alpha_2^0)$ and
$(\alpha_2^0,\alpha_1^0)$ are always isomorphic (as oriented fibrations)
for $|\alpha_1^0|=|\alpha_2^0|=1$; for $|\alpha_1^0|\neq|\alpha_2^0|$
they are isomorphic precisely when $q^2\equiv 1$ mod~$p$. There
is never an orientation-reversing isomorphism between the
two Seifert fibrations.
\qed
\end{prop}

\vspace{2mm}

\textbf{(C)} Finally, we need to check potential isomorphisms coming
from replacing $(\alpha_1^0,\alpha_2^0)$ by $(\alpha_2^0,-\alpha_1^0)$.
The necessary condition on multiplicities now becomes
\begin{equation*}
\label{eqn:u-C}
\tag{u-C}
u:=\gcd(p,s\alpha_1^0-\alpha_2^0)=\gcd(p,s\alpha_2^0+\alpha_1^0),\]
and we can write
\[ \uM=M\bigl(0;(\alpha_2,\ubeta_1),(-\alpha_1,\ubeta_2)\bigr).\]

\vspace{1mm}

\textbf{(C.1)} If $|\alpha_1^0|\neq|\alpha_2^0|$, then $M=\pm\uM$
holds precisely when there is an $\ell\in\Z$ such that
\[ \pm\ubeta_1=\beta_2+\ell\alpha_2,\;\;\; \mp\ubeta_2=\beta_1-\ell\alpha_1.\]
Computing as in the other cases, one finds that only $M=-\uM$ is
possible. By an argument as in \textbf{(B.1)} one sees
that the divisibility condition becomes $q^2\equiv -1$ mod~$p$.
This condition implies (\ref{eqn:u-C}), for
\[ s(s\alpha_1^0-\alpha_2^0)\equiv-(s\alpha_2^0+\alpha_1^0)\;
\text{mod}\; p.\]

\vspace{1mm}

\textbf{(C.2)} If $|\alpha_1^0|=|\alpha_2^0|=1$, then in notation
as earlier we have the additional possibility that
\[ \pm\varepsilon\ubeta_1=\beta_1+\ell\alpha_1,\;\;\;
\mp\varepsilon\ubeta_2=\beta_2-\ell\alpha_2.\]
Again, only $M=-\uM$ might happen. The divisibility condition,
with arguments as in \textbf{(A.1)}, becomes $p\in\{1,2\}$,
which is included in the condition $q^2\equiv -1$ mod~$p$ from \textbf{(C.1)}.
Alternatively, observe that this case \textbf{(C.2)} actually
coincides with~\textbf{(A.2)}.

\vspace{1mm}

In conclusion, we have the following.

\begin{prop}
\label{prop:C}
The two Seifert fibrations coming from $(\alpha_1^0,\alpha_2^0)$
and $(\alpha_2^0,-\alpha_1^0)$ are never isomorphic as
oriented fibrations. An orientation-reversing isomorphism exists
if and only if $q^2\equiv-1$ mod~$p$.
\qed
\end{prop}

Write symbolically (e) for the Seifert fibration obtained
from the original choice $(\alpha_1^0,\alpha_2^0)$, and
(a), (b), (c) for the fibrations obtained via the
alternative choices described above. It suffices to understand
the potential isomorphisms under these three choices when we
observe that (b) is obtained from (a) via the process described
in~\textbf{(C)}, (c) from (a) via~\textbf{(B)}, and (c) from (b)
via~\textbf{(A)}. Thus, if $|\alpha_1^0|=|\alpha_2^0|$=1, we have
\[ (e)=(b),\;\;\text{distinct from}\;\; (a)=(c)
\;\;\text{if}\;\; q^2\not\equiv-1\;\text{mod}\; p,\]
and
\[ (e)=-(a)=(b)=-(c)\;\;\text{if}\;\; q^2\equiv-1\;\text{mod}\; p.\]

If $|\alpha_1^0|\neq|\alpha_2^0|$, we find, with the cases numbered
as in Theorem~\ref{thm:equivalences}~(ii),
\begin{enumerate}
\item[(1)] (e), (a), (b), (c) distinct.
\item[(2)] $(e)=(b)$, distinct from $(a)=(c)$.
\item[(3)] $(e)=-(c)$, distinct from $(a)=-(b)$.
\item[(4)] $(e)=-(a)=(b)=-(c)$.
\end{enumerate}
This concludes the proof of Theorem~\ref{thm:equivalences}.

The numbering in the following examples corresponds to the
numbering in the theorem. Whenever we write $M_1=^*M_2$ in these
lists of examples, we mean to say that $M_1$ and $M_2$ are
distinct Seifert fibrations of the lens space in question, but
$M_1$ is orientation-reversingly isomorphic to~$M_2$
(that is, $M_1=-M_2$).

\begin{ex}
(i) The Seifert fibrations of $L(3,2)$ with $\{|\alpha_1^0|,|\alpha_2^0|\}
=\{1,1\}$ are
\[ M\bigl(0;(3,-1),(3,2)\bigr)\;\;\;\text{and}\;\;\;
M\bigl(0;(1,-3)\bigr).\]

The Seifert fibrations of $L(5,2)$ with $\{|\alpha_1^0|,|\alpha_2^0|\}
=\{1,1\}$ are
\[ M\bigl(0;(5,4),(5,-3)\bigr)=^*M\bigl(0;(5,-4),(5,3)\bigr).\]

(ii) (1) The Seifert fibrations of $L(7,2)$ with $\{|\alpha_1^0|,|\alpha_2^0|\}
=\{5,2\}$ are
\[ M\bigl(0;(35,-2),(14,1)\bigr),\;\;\; M\bigl(0;(35,-8),(14,3)\bigr),\]
\[ M\bigl(0;(35,-22),(14,9)\bigr),\;\;\; M\bigl(0;(35,-3),(14,1)\bigr).\]

(2) The Seifert fibrations of $L(3,2)$ with $\{|\alpha_1^0|,|\alpha_2^0|\}
=\{5,3\}$ are
\[ M\bigl(0;(15,2),(9,-1)\bigr)\;\;\;\text{and}\;\;\;
M\bigl(0;(15,-7),(9,4)\bigr).\]

(3) The Seifert fibrations of $L(5,2)$ with $\{|\alpha_1^0|,|\alpha_2^0|\}
=\{3,2\}$ are
\[ M\bigl(0;(15,2),(10,-1)\bigr)=^*M\bigl(0;(15,-2),(10,1)\bigr),\]
\[ M\bigl(0;(15,4),(10,-3)\bigr)=^*M\bigl(0;(15,-4),(10,3)\bigr).\]

(4) The Seifert fibrations of $L(2,1)$ with $\{|\alpha_1^0|,|\alpha_2^0|\}
=\{ 5,3\}$ are
\[ M\bigl(0;(5,-1),(3,1)\bigr)=^*M\bigl(0;(5,1),(3,-1)\bigr).\]
\end{ex}

\begin{rem}
The examples for (ii) were chosen such that the base orbifold is always
the same, even if the fibrations are not isomorphic. If one
of the three conditions (\ref{eqn:u-A}), (\ref{eqn:u-B}), (\ref{eqn:u-C})
is violated, which may happen in cases (i) and (ii)~(1)--(3)
of Theorem~\ref{thm:equivalences}, 
the base orbifolds will be different. The Seifert fibrations
of $L(3,2)$ in (i) are an instance of that; for an example
in (ii) take $L(7,2)$ and $(\alpha_1^0,\alpha_2^0)=(5,3)$.
\end{rem}

\begin{ex}
In \cite[Lemma~3.4]{lang16}, the Seifert fibrations of
$L(p,1)$ with two singular fibres of equal multiplicity were
determined by an \emph{ad hoc} argument. The case $p=2$
was relevant in~\cite{fls16}. Here we show how this
classification fits into our general scheme.

A complete list of Seifert fibrations with singular fibres of
equal multiplicity, up to orientation-preserving
bundle isomorphism, comes from $\alpha_1^0=1$ and $\alpha_2^0=\pm 1$.
We refer to the two choices of $\alpha_2^0$ as the cases $(\pm)$.
We may choose $s=1$ (and $r=0$). Then
\[ \alpha=\frac{p}{\gcd(p,1\mp 1)}=
\begin{cases}
1   & \text{in case $(+)$},\\
p   & \text{in case $(-)$ and $p$ odd},\\
p/2 & \text{in case $(-)$ and $p$ even.}
\end{cases}\]
In the sequel, we always keep the order of these three cases. Then
\[ (\alpha_1,\alpha_2)=
\begin{cases}
(1,1), &\\
(p,-p), &\\
(p/2,-p/2). &
\end{cases}\]
Next, we have
\[ \alpha_1'=\frac{1\mp 1}{\gcd(p,1\mp 1)}=
\begin{cases}
0,&\\
2,&\\
1.&
\end{cases}\]
This allows us to choose
\[ (\beta_1,\beta_1')=
\begin{cases}
(0,1),&\\
((p-1)/2,1),&\\
(-1,0).&
\end{cases}\]
Then
\[ \beta_2=-\beta_1+p\beta_1'=
\begin{cases}
p,&\\
(1-p)/2+p=(1+p)/2,&\\
1.&
\end{cases}\]
Thus, the Seifert fibrations of $L(p,1)$ with two singular fibres
of equal multiplicity are
\[ M\bigl(0;(1,0),(1,p)\bigr)=M\bigl(0;(1,p)\bigr)\]
and, for $p$ odd,
\[ M\Bigl(0;\Bigl(p,\frac{p-1}{2}\Bigr),\Bigl(-p,\frac{1+p}{2}\Bigr)\Bigr)=
-M\Bigl(0;\Bigl(p,\frac{1-p}{2}\Bigr),\Bigl(p,\frac{1+p}{2}\Bigr)\Bigr);\]
for $p$ even,
\[ M=\Bigl(0;\Bigl(\frac{p}{2},-1\Bigr),\Bigl(-\frac{p}{2},1\Bigr)\Bigr)=
-M\Bigl(0;\Bigl(\frac{p}{2},1\Bigr),\Bigl(\frac{p}{2},1\Bigr)\Bigr).\]
\end{ex}
\section{The model fibrations}
\label{section:model}
We return to the model fibrations introduced in
Section~\ref{subsection:model}. We determine the Seifert invariants of these
fibrations by placing them in the context of the algorithm in
Theorem~\ref{thm:algorithm}. In particular, the discussion in this section
will show the following theorem.

\begin{thm}
\label{thm:model}
Every Seifert fibration, over an orientable base,
of a lens space $L(p,q)$, $p>0$,
is isomorphic to one of the model fibrations.
\end{thm}

We first consider the simple case of $L(1,0)=S^3$, cf.\
\cite[Satz~11]{seif33}.

\begin{prop}
A complete list of the Seifert fibrations of $S^3$ is provided by
\[ M\bigl(0;(\alpha_1,\beta_1),(\alpha_2,\beta_2)\bigr),\]
where $\alpha_1,\alpha_2$ is any pair of coprime natural numbers
with $\alpha_1\geq\alpha_2$,
and $\beta_1,\beta_2$ is a pair of integers with
$0\leq\beta_1<\alpha_1$ such that
$\alpha_1\beta_2+\beta_1\alpha_2=1$.
These Seifert fibrations are realised by the
model $S^1$-action
\[ \theta(z_1,z_2)=(\rme^{\rmi\alpha_2\theta}z_1,
\rme^{\rmi\alpha_1\theta}z_2). \]
\end{prop}

\begin{proof}
By Theorem~\ref{thm:jane4.4}, the condition for
$M\bigl(0;(\alpha_1,\beta_1),(\alpha_2,\beta_2)\bigr)$ to equal $S^3$ is
$\alpha_1\beta_2+\beta_1\alpha_2=1$, i.e.\ that $\alpha_1,\alpha_2$
be coprime. Without loss of generality we may assume
$\alpha_1\geq\alpha_2$. For given $\alpha_1,\alpha_2$,
the pair $(\beta_1,\beta_2)$ is unique up to adding $(k\alpha_1,-k\alpha_2)$,
$k\in\Z$. By (S2) this does not affect the Seifert fibration
up to isomorphism,
and by applying this equivalence we can choose a unique representative
of the Seifert invariants where $0\leq\beta_1<\alpha_1$.

In particular, this argument shows that any given Seifert fibration
of $S^3$ is determined by the multiplicities $\alpha_1,\alpha_2$ of
the singular fibres, and hence it is isomorphic
to one of the model fibrations.
\end{proof}

From now on we consider a lens space $L(p,q)$ with $p>q>0$ and
$\gcd(p,q)=1$, thought of as a quotient of $S^3$ under the
$\Z_p$-action~(\ref{eqn:Z_p}). We begin by considering the
$S^1$-action
\[ \theta(z_1,z_2)=(\rme^{\rmi k_1\theta}z_1,\rme^{\rmi k_2\theta}z_2)\]
on~$S^3$. Define a decomposition of $S^3=\wtV_1\cup\wtV_2$ into two solid
tori by setting
\[ \wtV_i:=\bigl\{(z_1,z_2)\in S^3\co |z_i|^2\leq 1/2\bigr\},\;\;
i=1,2.\]
On the $2$-tori $\partial\wtV_i=:\wtT_i$ we define meridians
$\tmu_i$ and longitudes $\tlambda_i$ by
\begin{align*}
\tmu_1(t)     & =\frac{1}{\sqrt{2}}(\rme^{\rmi t},1),&
\tlambda_1(t) & =\frac{1}{\sqrt{2}}(\rme^{\rmi st},\rme^{\rmi t}),\\
\tmu_2(t)     & =\frac{1}{\sqrt{2}}(1,\rme^{\rmi t}),&
\tlambda_2(t) & =\frac{1}{\sqrt{2}}(\rme^{\rmi t},\rme^{\rmi qt}),
\end{align*}
where $t$ always runs from $0$ to~$2\pi$, and $r,s$ are as
in Section~\ref{subsection:lens-surgery}. Our choice of longitudes is
explained by the fact that these curves are invariant under
the $\Z_p$-action and hence will descend to longitudes on the
two solid tori making up $L(p,q)$.

A regular fibre of the Seifert fibration of $S^3$ lying inside
the $2$-torus $\wtT_1=\wtT_2$ can be parametrised as
\[ \tildeh(t)=\frac{1}{\sqrt{2}}(\rme^{\rmi k_1t},\rme^{\rmi k_2t}),\;\;
t\in[0,2\pi].\]
In the homology of $\wtT_1=\wtT_2$ we have
\[ \tildeh=-(sk_2-k_1)\tmu_1+k_2\tlambda_1\]
and
\[ \tildeh=-(qk_1-k_2)\tmu_2+k_1\tlambda_2,\]
respectively. Comparing this with the gluing map~(\ref{eqn:seifert})
we see that this amounts to
\[ (\talpha_1,\talpha_1')=(k_2,sk_2-k_1)\]
and
\[ (\talpha_2,\talpha_2')=(k_1,qk_1-k_2),\]
from which one can easily determine the Seifert invariants
of the Seifert fibration of~$S^3$.

We now consider the projection
\[ \mfp\co S^3\longrightarrow S^3/\Z_p=L(p,q)\]
to the quotient. The images $V_i:=\mfp(\wtV_i)$ are again solid tori, with
meridian and longitude given by
\[ \mu_i=\mfp\circ\tmu_i\;\;\;\text{and}\;\;\;
\lambda_i=\mfp\circ\tlambda_i|_{[0,2\pi/p]}.\]
Notice that $p\lambda_2-q\mu_2=\mu_1$ in the homology of $T_i$,
which accords with~(\ref{eqn:lens-surgery}).

What are the Seifert invariants of the induced Seifert fibration of
$L(p,q)?$ Obviously, the length of the singular fibres $S^1\times\{0\}$
and $\{0\}\times S^1$ on~$S^3$, which was $|2\pi/k_1|$ and $|2\pi/k_2|$,
respectively, is reduced to $|2\pi/pk_1|$ and $|2\pi/pk_2|$, since these fibres
are invariant under the $\Z_p$-action (and the action on the fibre
is free). If the regular fibres were freely permuted by the $\Z_p$-action,
their length would remain equal to~$2\pi$, and the new multiplicities
of the singular fibres would be $|pk_1|$, $|pk_2|$. In general, however,
a subgroup of $\Z_p$ will leave a regular fibre invariant
(and act freely on it).

We now determine this subgroup and thence deduce the Seifert invariants.
A particular case of this analysis was carried out in
\cite[Proposition~7.6]{gego15}.

The following notation is chosen with prescience. Write $u$ for the
number of elements in $\Z_p$ whose action leaves the regular fibre
$\tildeh$ invariant, that is, the number of $\ell\in\{1,\ldots,p\}$
with $\rme^{2\pi\rmi\ell/p}\tildeh=\tildeh$ (as a set). Then the
regular fibre on the quotient $S^3/\Z_p$ is parametrised by
\[ h=\mfp\circ\tildeh|_{[0,2\pi/u]}.\]
Notice that in the homology of $T_1$ we have
\[ uh=-(sk_2-k_1)\mu_1+k_2p\lambda_1.\]

The following lemma shows that if we take
$k_1:=\alpha_2^0$ and $k_2:=\alpha_1^0$ in the construction above,
we have precisely the set-up of the algorithm in Theorem~\ref{thm:algorithm}.
Indeed, the formula above then becomes $h=\alpha_1\lambda_1-
\alpha_1'\mu_1$ as in~(\ref{eqn:seifert}).
This allows one to read off the Seifert invariants and
concludes the proof of Theorem~\ref{thm:model}.

\begin{rem}
Notice that the circle $S^1\times\{0\}\subset S^3$, which is the
singular fibre of multiplicity $k_1$ in the Seifert fibration
of~$S^3$, is the spine of the solid torus~$\wtV_2$; the fibre
$\{0\}\times S^1$ of multiplicity $k_2$ is the spine of~$\wtV_1$.
This explains why the indices of the $k_i$ and
the $\alpha_i^0$ are interchanged.
\end{rem}

\begin{lem}
$u=\gcd(p,sk_2-k_1)$.
\end{lem}

\begin{proof}
In order to determine $u$, we consider a point on the fibre $h$
and its translates under the $\Z_p$-action. Then $u$ is the
number of these translates that lie again on~$h$.

We pass to the universal cover $\R^2$ of $T_1$ and lift the
$S^1$-action and the $\Z_p$-action on $T_1$ to an $\R$-action
and a $\Z$-action, respectively. With respect to the euclidean
metric on~$\R^2$, the $\R$-action is given by
\[ (x,y)\longmapsto (x+k_1t,y+k_2t);\]
the $\Z$-action is given by
\[ (x,y)\longmapsto \Bigl(x+\frac{1}{p},y+\frac{q}{p}\Bigr).\]
Figure~\ref{figure:slopes} shows the situation for
$p=6$, $q=5$, $k_1=3$ and $k_2=1$. For simplicity, we assume
$q>k_2/k_1>0$ for the remainder of the argument;
the other three cases $q=k_2/k_1$, $q<k_2/k_1$ and
$k_2/k_1<0$ are analogous.

\begin{figure}[h]
\labellist
\small\hair 2pt
\pinlabel $(1,q)$ [tl] at 277 554
\pinlabel $y$ [r] at 158 572
\pinlabel $x$ [t] at 519 14
\pinlabel $S$ [r] at 241 416
\pinlabel $\hat{h}$ [b] at 339 80
\endlabellist
\centering
\includegraphics[scale=0.4]{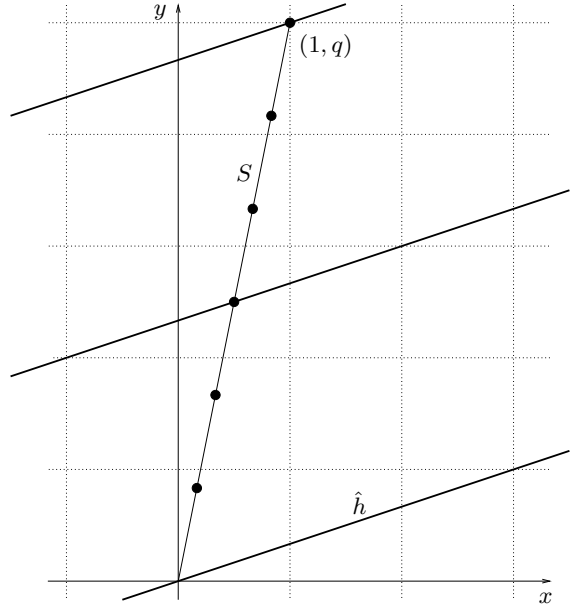}
  \caption{Determining $u$ geometrically.}
  \label{figure:slopes}
\end{figure}

Set
\[ P:=\Bigl\{\Bigl(\frac{\ell}{p},\frac{\ell q}{p}\Bigr)\co
\ell=1,\ldots,p\Bigr\}\]
and consider the lift
\[ \hat{h}:=\bigl\{(k_1t,k_2t)\co t\in\R\bigr\}\]
of $h$ through the point $(0,0)$. The set of all lifts of $h$
is given by
\[ H:=\hat{h}+\Z^2.\]
Then $u$ equals the number of intersection points of $P$ with~$H$,
that is,
\[ u=\#(P\cap H).\]
Notice that the vertical distance between adjacent lines in the family $H$
is $1/k_1$, the horizontal distance is $1/k_2$.

In order to determine this number~$u$, we first consider the
line segment
\[ S:=\bigl\{(t,qt)\co t\in(0,1]\bigr\}, \]
which is subdivided by the points in $P$ into $p$ parts of
equal length. Write $X$ for the closed line segment
joining the points $(0,0)$ and $(1,0)$, and $Y$ for the line
segment joining $(1,0)$ with $(1,q)$.
Then the number of intersection
points of $H$ with the half-open line segment $S$ equals the
number of times $H$ intersects $Y$ minus the number of times
it intersects~$X$. Thus,
\begin{eqnarray*}
\#(S\cap H)
& = & \#(Y\cap H)-\#(X\cap H)\\
& = & qk_1-k_2.
\end{eqnarray*}
So the line segment $S$ is cut by $H$ into $qk_1-k_2$
parts of equal length, and $u$ is the number of these
division points that lie in~$P$.

The first division points along $S$ that coincide are characterised
by the existence of coprime natural numbers $n_1,n_2$ such that
\[ \frac{n_1}{p}=\frac{n_2}{qk_1-k_2},\]
which is equivalent to $n_1(qk_1-k_2)=n_2p=\lcm(p,qk_1-p)$.
It follows that
\[ u=\frac{p}{n_1}=\gcd(p,qk_1-k_2)=\gcd(p,sk_2-k_1),\]
as claimed.
\end{proof}

\begin{rem}
In the standard model from Section~\ref{subsection:model},
the diffeomorphism giving the isomorphism in case \textbf{(A)} is
given by $(z_1,z_2)\mapsto(\oz_1,z_2)$, in case \textbf{(B)} by
$(z_1,z_2)\mapsto(z_2,z_1)$, and in case \textbf{(C)} by $(z_1,z_2)\mapsto
(\oz_2,z_1)$.
\end{rem}
\begin{ack}
We thank the referee for useful comments that helped to improve
the exposition.
\end{ack}

\end{document}